\documentclass[12pt]{amsart}

\usepackage{amsfonts, amsthm, amsmath}
\allowdisplaybreaks[4]
\usepackage{pgfplots}

\usepackage{rotating}

\usepackage{tikz}
\usepackage{lmodern}

\usetikzlibrary{decorations.pathmorphing}

\usepackage{xcolor}

\usepackage{graphics}

\usepackage{amssymb}

\usepackage{amscd}

\usepackage[latin2]{inputenc}

\usepackage{t1enc}

\usepackage[mathscr]{eucal}

\usepackage{indentfirst}

\usepackage{enumitem}

\usepackage{enumitem}

\usepackage{graphicx}

\usepackage{graphics}

\usepackage{pict2e}

\usepackage{mathrsfs}

\usepackage{comment}

\usepackage{enumitem}

\usepackage{enumerate}
\usepackage[pagebackref]{hyperref}
\hypersetup{colorlinks=true}
\usepackage{color}
\usepackage{epic}
\usepackage{framed}
\usepackage{mathabx}
\usepackage{booktabs}
\usetikzlibrary{decorations.pathreplacing,calc}
\usepackage{makecell}
\newcolumntype{V}{!{\vrule width 2pt}}

\numberwithin{equation}{section}

\def\blue{\textcolor{blue}}
\def\red{\textcolor{red}}

\theoremstyle{plain}

\newtheorem{theorem}{Theorem}[section]

\newtheorem{lemma}[theorem]{Lemma}

\def\S{\mathfrak{S}}
\def\N{\mathbb{N}}
\def\B{\mathfrak{B}}
\def\W{\mathfrak{W}}
\def\R{\mathfrak{R}}
\def\P{\mathcal{P}}
\def\Sk{\mathcal{S}}

\def\Mo{\mathcal{M}}
\def\L{\mathcal{L}}

\begin{document}
\title[Bijections around Springer numbers]{Bijections around Springer numbers}

\author[S. Chen]{Shaoshi Chen}
\address[Shaoshi Chen]{KLMM, Academy of Mathematics and Systems Science, Chinese Academy of Sciences, Beijing 100190, P.R. China}
\email{schen@amss.ac.cn}

\author[Y. Li]{Yang Li}
\address[Yang Li]{Research Center for Mathematics and Interdisciplinary Sciences, Shandong University, \& Frontiers Science Center for Nonlinear Expectations, Ministry of Education, Qingdao 266237, P.R. China}
\email{202421349@mail.sdu.edu.cn}

\author[Z. Lin]{Zhicong Lin}
\address[Zhicong Lin]{Research Center for Mathematics and Interdisciplinary Sciences,  Shandong University \& Frontiers Science Center for Nonlinear Expectations, Ministry of Education, Qingdao 266237, P.R. China}
\email{linz@sdu.edu.cn}

\author[S.H.F. Yan]{Sherry H.F. Yan}
\address[Sherry H.F.  Yan]{Department of Mathematics,
Zhejiang Normal University, Jinhua 321004, P.R. China}
\email{hfy@zjnu.cn}

\date{\today}

\begin{abstract}
Arnol'd proved in 1992 that Springer numbers enumerate the Snakes, which are  type $B$ analogs of alternating permutations.  Chen, Fan and Jia in 2011 introduced   the labeled ballot paths and established a ``hard'' bijection with  snakes. Callan conjectured in 2012 and Han--Kitaev--Zhang proved recently that rc-invariant alternating permutations are counted by Springer numbers. Very recently, Chen--Fang--Kitaev--Zhang investigated multi-dimensional permutations and proved that weakly increasing $3$-dimensional permutations are also counted by Springer numbers. In this work, we construct a sequence of ``natural'' bijections linking the above four combinatorial  objects.
\end{abstract}

\keywords{Springer numbers; Snakes; Alternating permutations; Labeled ballot paths}

\maketitle

\section{Introduction}

The {\em Euler numbers} $E_n$ are the coefficients of the  Taylor expansion:
$$
\tan(x)+\sec(x)=\sum_{n\geq0} E_n\frac{x^n}{n!}=1+x+1\frac{x^2}{2!}+2\frac{x^3}{3!}+5\frac{x^4}{4!}+16\frac{x^5}{5!}+61\frac{x^6}{6!}+\cdots.
$$
Let $\S_n$ denote the set of permutations of $[n]:=\{1,2,\ldots,n\}$. A permutation $\pi\in\S_n$ possessing  the down-up (or alternating) property
\begin{equation}\label{down:up}
\pi_1>\pi_2<\pi_3>\pi_4<\cdots
\end{equation}
is called an {\em alternating permutation}.
It is a classical result in Enumerative Combinatorics~\cite[Prop.~1.6.1]{EC1} that $E_n$ counts   alternating   permutations of length $n$.

The {\em Springer numbers} $S_n$ defined by
$$
\frac{1}{\cos(x)-\sin(x)}=\sum_{n\geq0} S_n\frac{x^n}{n!}=1+x+3\frac{x^2}{2!}+11\frac{x^3}{3!}+57\frac{x^4}{4!}+361\frac{x^5}{5!}+2763\frac{x^6}{6!}+\cdots
$$
can be considered as type $B$ analogs of the Euler numbers. Let $\B_n$ denote the set of {\em signed permutations} of length $n$, whose elements are those of the form $\pm\pi_1\pm\pi_2\cdots\pm\pi_n$ with $\pi_1\pi_2\cdots\pi_n\in\S_n$. For convenience, we write $-n$ by $\bar{n}$ for each positive integer $n$. A signed permutation $\pi\in\B_n$ satisfying $\pi_1>0$ and the down-up property~\eqref{down:up}
 is called a {\em snake} (of type $B_n$).  In his study of connections between combinatorics of the Coxeter groups and singularities of smooth functions, Arnol'd~\cite{Ar} proved that $S_n$ enumerates snakes of length $n$. For example, $S_3$ counts the following 11 snakes of length $3$
 $$
1 \bar23, 1\bar32, 1\bar3\bar2, 213, 2\bar13, 2\bar31, 2\bar3\bar1, 312, 3\bar12, 3\bar21, 3\bar2\bar1.
 $$

Besides snakes of length $n$, there are various interesting combinatorial interpretations for the Springer number $S_n$, among which are
\begin{itemize}
\item labeled ballot paths of $n$ steps, introduced and proved by Chen, Fan and Jia~\cite{CFJ};
\item rc-invariant alternating permutations of length $2n$, proved recently by Han, Kitaev and Zhang~\cite{HKZ}, verifying an observation by Callan in 2012;
\item weakly increasing $3$-dimensional permutations, found and proved very recently by Chen, Fang, Kitaev and Zhang~\cite{Chen0}.
\end{itemize}
 Chen, Fan and Jia~\cite{CFJ}  established a bijection between  the labeled ballot paths and     snakes, which is ``hard'' in the sense that the proof of its bijectivity is intricate. On the other hand, the proofs of the latter two assertions in~\cite{HKZ} and~\cite{Chen0} are not bijective. In this note, we construct a sequence of ``natural'' bijections linking the above four combinatorial  objects counted by Spring numbers (see Fig.~\ref{bij:spr}).

\begin{figure}
\centering
\begin{tikzpicture}[scale=0.3]

\node at (15,26.9) {Snakes};

\node at (3.1,16.6) {\small{Weakly increasing}};
\node at (3,15.1) {\small{3-D permutations}};

\draw[thick] (11,26) to (3,20);
\draw[thick] (11,26) to (10.5,25);
\draw[thick] (11,26) to (10,25.8);
\draw[thick] (3,20) to (4,20.2);
\draw[thick] (3,20) to (3.5,20.9);
\node at (6.2,23.3) {$\Phi$};

\draw[thick] (19,26) to (26,20);
\draw[thick] (19,26) to (20,25.7);
\draw[thick] (19,26) to (19.2,25.1);
\draw[thick] (26,20) to (25,20.2);
\draw[thick] (26,20) to (25.5,21);
\node at (23,23.5) {$\psi$};

\red{\node at (15,22.8) {Springer};
\node at (15,20.8) {Numbers};}

\draw[thick,red] (15,21.8) ellipse (5 and 2.5);
\draw[thick,dotted] (15,26.8) ellipse (3 and 1.5);
\draw[thick,dotted] (3,16) ellipse (5.8 and 2.7);
\draw[thick,dotted] (15,16) ellipse (4.2 and 2.2);
\node at (15,17) {Labeled};
\node at (15,15.4) {ballot paths};
\draw[thick,dotted] (29,16) ellipse (5.8 and 2.7);
\node at (29,17) {rc-invariant};
\node at (29,15.5) {alternating perms};
\node at (21,15.9) {$\longleftrightarrow$};
\node at (21,16.8) {$\Psi$};

\end{tikzpicture}
\caption{Bijections around Springer numbers\label{bij:spr}}
\end{figure}
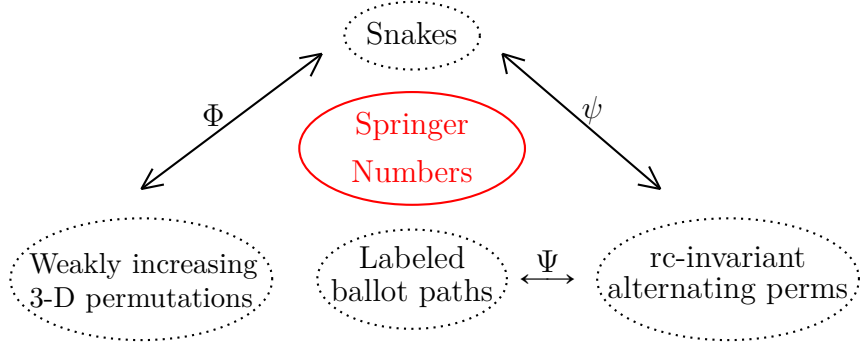

\section{Bijecting weakly increasing $3$-dimensional permutations with snakes }

 A pair $(\sigma,\pi)\in\S_n^2$ is called a  {\em weakly increasing $3$-dimensional permutation\footnote{ Since the pair $(\sigma,\pi)$ can be written as $3$-dimensional array
 $$
 \begin{bmatrix}
1 & 2& \cdots & n \\
\sigma_1 & \sigma_2 & \cdots& \sigma_n \\
\pi_1 & \pi_2 & \cdots& \pi_n
\end{bmatrix},
 $$
 it is referred to as $3$-dimensional permutation in~\cite{Chen0}.
 }} (3-WIP for short) in~\cite{Chen0} if
 $$
 \max(\sigma_1,\pi_1)\leq \max(\sigma_2,\pi_2)\leq \cdots\leq \max(\sigma_n,\pi_n).
 $$
Let $\W_n$ be the set of all 3-WIPs of length $n$. It is convenience to write $(\sigma,\pi)\in\W_n$ in 2-array as
 $$
 \begin{bmatrix}
\sigma_1 & \sigma_2 & \cdots& \sigma_n \\
\pi_1 & \pi_2 & \cdots& \pi_n
\end{bmatrix}.
$$
For example,
 \begin{equation}\label{exm1}
 \begin{bmatrix}
1&5&2&6&7&3&8&9&4 \\
2&5&6&3&1&7&8&4&9
\end{bmatrix}\in\W_9.
 \end{equation}

 For $\sigma\in\S_n$, denote by $\sigma^{-1}$ its inverse in $\S_n$. A letter $k$, $2\leq k\leq n$, is called a {\em cycle peak} of $\sigma$ if $\sigma^{-1}_k<k>\sigma_k$. For example, with $\sigma=2\,6\,7\,9\,5\,3\,1\,8\,4$ written in its cycle form as $(1,2,6,3,7)(5)(4,9)(8)$, its cycle peaks are $6$, $7$ and $9$.
 For our purpose, we need to recall Foata's ``transformation fondamentale'' $f:\S_n\to\S_n$ (see~\cite[Page~23]{EC1}), which is constructed in two steps:
\begin{itemize}
\item Write the permutation $\sigma$ in its {\it standard cycle form} by requiring:
(a) each cycle has its largest letter in the leftmost position; (b) the cycles are listed from left to right in increasing order of their largest letters.
\item Erasing all the parentheses results in the permutation $f(\sigma)$.
\end{itemize}
Continuing with the running example $\sigma$, its standard cycle form is $(5)(7,1,2,6,3)(8)(9,4)$ and we get
$o(\sigma)=5\,7\,1\,2\,6\,3\,8\,9\,4$ after erasing all the parentheses. It is plain to see that Foata's bijection $f$ transforms cycle peaks of $\sigma$ to  left peaks of $\pi=f(\sigma)$, where a letter $\pi_i$ of $\pi$ is a {\em left peak}  if $\pi_{i-1}<\pi_i>\pi_{i+1}$ with the convention $\pi_0=0$ and $\pi_{n+1}=+\infty$.

We are ready to construct the bijection $\Phi$ from $\W_n$ to $\Sk_n$, the set of snakes of length $n$. Given $(\sigma,\pi)\in\W_n$, we perform the following three steps to get the snake $\Phi(\sigma,\pi)$:
\begin{itemize}
\item[(1)] Construct a permutation $\tau$ such that $\tau(\sigma_i)=\pi_i$ for each $i$ with its cycle peak $k$ is hatted if and only if $k=\sigma_{\ell}=\pi_{\ell+1}$ for some $\ell$. Take $(\sigma,\pi)$ in~\eqref{exm1} as example, we have $\tau=(5)(\hat7,1,2,6,3)(8)(\hat9,4)$, written in its standard cycle form.
\item[(2)] Write $\tau$ in its standard cycle form and erase the parentheses to get $\widetilde\tau=f(\tau)$, which is a permutation in $\S_n$ with some of its left peaks hatted. Continuing with the running example, we have $\widetilde\tau=f(\tau)=5\,\hat7\,1\,2\,6\,3\,8\,\hat9\,4$.
\item[(3)] A letter $\widetilde\tau_i$ of $\widetilde\tau$ is a {\em right valley}  if $\widetilde\tau_{i-1}>\widetilde\tau_i<\widetilde\tau_{i+1}$ with the convention $\widetilde\tau_0=0$ and
$\widetilde\tau_{n+1}=+\infty$. We associate each right valley of $\widetilde\tau$  with the closest left peak  to its left in  $\widetilde\tau$. Now define the snake $\Phi(\sigma,\pi)$ from $\widetilde\tau$ by first removing hats and then putting bars over the letters in two situations: a non-right-valley in even position or a right valley whose associated left peak in $\widetilde\tau$ is hatted.  Continuing with the running example, we have $\Phi(\sigma,\pi)=5\,\bar7\,\bar1\,\bar2\,6\,3\,8\,\bar9\,\bar4$.
\end{itemize}

\begin{theorem}
The mapping  $\Phi: \W_n\rightarrow\Sk_n$  is a bijection.
\end{theorem}
\begin{proof}
We need to verify that $\Phi(\sigma,\pi)=p\in\B_n$ is a snake. As $\widetilde\tau_1$ is not a right valley,  $p_1>0$. We distinguish two cases.
\begin{itemize}
\item If $i$ is odd and $i<n$, then we need to show that $p_i>p_{i+1}$. If $\widetilde\tau_i<\widetilde\tau_{i+1}$, then $p_{i+1}=-\widetilde\tau_{i+1}$ and so $p_i>p_{i+1}$.   If $\widetilde\tau_i>\widetilde\tau_{i+1}$, then $p_{i}=\widetilde\tau_i$ is positive  and so $p_i>p_{i+1}$.

\item If $i$ is even and $i<n$, then we need to show that $p_i<p_{i+1}$. If $\widetilde\tau_i<\widetilde\tau_{i+1}$, then $p_{i+1}=\widetilde\tau_{i+1}$ and so $p_i<p_{i+1}$.   If $\widetilde\tau_i>\widetilde\tau_{i+1}$, then $p_{i}=-\widetilde\tau_i$  and so $p_i<p_{i+1}$.
\end{itemize}
It follows that $p$ is a snake and $\Phi$ is well-defined.

It is plain to see that steps (2) and (3) in the construction of $\Phi$ are  reversible. To see that step (1) is reversible, we define its inverse explicitly. Given $\tau\in\S_n$ with some of its cycle peaks hatted, form the 3-array
$$
T=
\begin{bmatrix}
1 & 2 & \cdots& n \\
\tau_1 & \tau_2 & \cdots& \tau_n\\
c_1 & c_2 & \cdots& c_n
\end{bmatrix},
$$
where $c_i=\max(i,\tau_i)$ for each $i$. Observe that $k$ appears twice in  the bottom row of $T$ iff $k$ is a cycle peak in $\tau$. Rearrange the columns of $T$ in weakly increasing order of their bottom values by requiring that whenever two columns $(j,\tau_j,c_j)$ and $(k,\tau_k,c_k)$ with $j<k$ and $c_j=c_k=k$, then put $(k,\tau_k,c_k)$ before $(j,\tau_j,c_j)$ iff $k$ is a cycle peak in $\tau$ with hat. Finally, removing the bottom row to retrieve the desired 3-WIP $(\sigma,\pi)$. This shows that step (1) is reversible and so $\Phi$ is a bijection.
\end{proof}

 \section{Bijecting snakes with labeled ballot paths}

 For $\pi\in\S_n$, the {\em complement} of $\pi$ is
 $$
 \pi^c:=(n+1-\pi_1)(n+1-\pi_2)\cdots(n+1-\pi_n),
 $$
 while the {\em reverse} of $\pi$ is
 $$
 \pi^r:=\pi_n\pi_{n-1}\cdots\pi_1.
 $$
A permutation $\pi$ is said to be {\em rc-invariant} if $\pi^{rc}=\pi$.  For instance, the permutation $\pi=41352$ is rc-invariant. The set of rc-invariant alternating permutations of length $2n$ is denoted by $\R_n$. This class of permutations was considered by Callan in the OEIS~\cite{oeis}, where he suspected that $|\R_n|=S_n$. His conjecture has been confirmed recently by Han, Kitaev and Zhang~\cite{HKZ}. With rc-invariant alternating permutations as intermediate structure, we construct a ``new'' bijection from snakes to labeled ballot paths.

First we construct a  bijection from $\Sk_n$ to $\R_n$, which is based on the following simple fact.
\begin{lemma}\label{lem:rc}
Given $\pi=\pi_1\pi_2\cdots\pi_{2n}\in\R_n$, then $\pi_n>n>\pi_{n+1}$ (resp.,~$\pi_n<n<\pi_{n+1}$) if $n$ is odd (resp.,~even).
\end{lemma}
\begin{proof}
Since $\pi$ is rc-invariant, we have $\pi_i+\pi_{2n+1-i}=2n+1$. The result then follows from the fact that $\pi$ is alternating.
\end{proof}
\begin{theorem}\label{thm:psi}
There exists a bijection $\psi: \Sk_n\rightarrow\R_n$.
\end{theorem}
\begin{proof}
Given a snake $\pi\in\Sk_n$, define $\tilde\pi=\tilde\pi_1\cdots\tilde\pi_n$ with
$$
\tilde\pi_i:= \begin{cases}
n+\pi_i,\quad&\text{if $\pi_i>0$};\\
n+1+\pi_i, \quad&\text{if $\pi_i<0$}.
\end{cases}
$$
Note that $\tilde\pi_k+\tilde\pi_l\neq2n+1$ for any $k\neq l$. Thus, we can define $\psi(\pi)$ to be the unique rc-invariant permutation whose first half (resp.,~later half) is $\tilde\pi^r$ (resp.,~$\tilde\pi$) if $n$ is odd (resp.,~even). It is clear that $\tilde\pi$ is alternating. As $\pi_1>0$, $\tilde\pi_1>n$ and so $\psi(\pi)$ is alternating by Lemma~\ref{lem:rc}. For example, if $\pi=215\bar4\bar3$, then $\tilde\pi=7\,6\,10\,2\,3$ and $\psi(\pi)=\red{3\,2\,10\,6\,7}\,4\,5\,1\,9\,8$; if $\pi=1\bar5\bar3\bar62\bar4$, then $\tilde\pi=7\,2\,4\,1\,8\,3$ and $\psi(\pi)=10\,5\,12\,9\,11\,6\,\red{7\,2\,4\,1\,8\,3}$.

It is plain to see that $\psi$ is reversible and so is a bijection.
\end{proof}

A {\em partial Motzkin path}  is a lattice path in the quarter plane $\N^2$ starting at $(0,0)$ and using three possible steps:
$$
\text{$(1,1)=U$ (up step), $(1,-1)=D$ (down step)\,\,\,and\,\,\,$(1,0)=H$ (horizontal step).}
$$
For a partial Motzkin path $P=p_1p_2\cdots p_n$ with $n$ steps, let $h_i(P)$ be the {\em height} of the $i$-th step of $P$:
$$
h_i(P):=|\{j\mid j<i, p_j=U\}|-|\{j\mid j<i, p_j=D\}|.
$$
A partial Motzkin path without horizontal steps is called a {\em ballot path}. As introduced in~\cite{CFJ}, a {\em labeled ballot path} of length $n$ is a pair $(P,w)$, where $P=p_1p_2\cdots p_n$ is a  ballot path of $n$ steps and $w=w_1w_2\cdots w_n\in\N^n$ is a weight function satisfying
$$
0\leq w_i\leq \begin{cases}
h_i(P),\quad&\text{if $p_i=U$};\\
h_i(P)-1, \quad&\text{if $p_i=D$}.
\end{cases}
$$
 Let $\P_n$ be the set of labeled ballot paths of length $n$.

Partial Motzkin paths ending at $x$-axis are the usual Motzkin paths.
A {\em two-colored Motzkin path} is  a Motzkin path whose  horizontal steps are colored by $\red{H}$ or $\blue{\tilde H}$. Denote by  $\Mo_n^{(2)}$ the set of all two-colored Motzkin paths of length $n$.
A {\em restricted Laguerre history} of length $n$ is a pair $(M,w)$, where $M=m_1m_2\cdots m_n\in\Mo_n^{(2)}$ and $w=w_1w_2\cdots w_n\in\N^n$ is a weight function satisfying
$$
0\leq w_i\leq \begin{cases}
h_i(M),\quad&\text{if $m_i=U, \red{H}$};\\
h_i(M)-1, \quad&\text{if $m_i=D, \blue{\tilde H}$}.
\end{cases}
$$
Let $\L_n$ denote the set of all restricted Laguerre histories of length $n$. In the following, we show that
the classical Foata--Zeilberger bijection~\cite{FZ} (see also~\cite{Cor}) $\Psi_{FZ}$ from $\S_n$ to $\L_n$ restricted to a bijection from $\R_n$ to $\P_n$.

We need to recall the Foata--Zeilberger bijection $\Psi_{FZ}$.
Given a permutation $\pi\in\S_n$, we use the convention
   $\pi_0=0$ and $\pi_{n+1}=+\infty$. For each $i\in[n]$, let $(\underline{31}2)_i(\pi)$ be the number of $\underline{31}2$-patterns in $\pi$ with $i$ representing the $2$, i.e.,
   $$(\underline{31}2)_i(\pi):=|\{k: k<j\text{ and } \pi_k<\pi_j=i<\pi_{k-1}\}|.$$
Define $\Psi_{FZ}(\pi)=(M,w)$, where for  $i\in[n]$ with $\pi_j=i$:
$$
  m_i=\left\{
  \begin{array}{ll}
  U&\quad\mbox{if $\pi_{j-1}>\pi_j<\pi_{j+1}$},  \\
 D&\quad\mbox{if $\pi_{j-1}<\pi_j>\pi_{j+1}$},  \\
  \red{H} &\quad\mbox{if $\pi_{j-1}<\pi_j<\pi_{j+1}$},\\
  \blue{\tilde{H}}&\quad\mbox{if $\pi_{j-1}>\pi_j>\pi_{j+1}$},
  \end{array}
  \right.
  $$
  and $w_i=(\underline{31}2)_i(\pi)$.
For example, if $\pi=431296857\in\S_9$, then
$$\Psi_{FZ}(\pi)=(U\red{H}\blue{\tilde H}DUU\red{H}DD,010000210).$$

The inverse algorithm $\Psi_{FZ}^{-1}$ building a permutation $\pi$ (in $n$ steps) from a Laguerre history $(M,w)\in\L_{n}$ can be described iteratively as:
\begin{itemize}
\item Initialization: $\pi=\diamond$;
\item At the $i$-th ($1\leq i\leq n$) step of the algorithm, replace the $(w_i+1)$-th $\diamond$ (from left to right) of $\pi$ by
$$
\begin{cases}
\,\diamond i\diamond&\quad \text{if $m_i=U$},\\
\, i\diamond&\quad \text{if $m_i=\red{H}$},\\
 \, i& \quad\text{if $m_i=D$},\\
 \, \diamond i&\quad \text{if $m_i=\blue{\tilde H}$};
\end{cases}
$$
\item The final permutation is obtained by removing  the last remaining $\diamond$.
\end{itemize}
For example, if $(M,w)=(U\red{H}\blue{\tilde H}DUU\red{H}DD,010000210)\in\L_9$, then
\begin{align*}
\pi&=\diamond\rightarrow \diamond1\diamond\rightarrow\diamond12\diamond\rightarrow \diamond312\diamond\rightarrow 4312\diamond\rightarrow 4312\diamond5 \diamond\\
&\quad\rightarrow 4312\diamond6\diamond5\diamond\rightarrow 4312\diamond6\diamond57\diamond\rightarrow 4312\diamond6857\diamond\rightarrow 431296857.
\end{align*}

From the inverse algorithm $\Psi_{FZ}^{-1}$,  one could check easily the following lemma that was known in~\cite{Cor}.

\begin{lemma}\label{lem:cor}
Suppose that $(M,w)\in\L_n$ and $\pi=\Psi_{FZ}^{-1}(M,w)$. Then for any $1\leq i\leq n$,
$$
(\underline{31}2)_i(\pi)+(2\underline{31})_i(\pi)=
\begin{cases}
h_i(M), &\quad\text{if $m_i=U, \red{H}$};\\
h_i(M)-1, &\quad\text{if $m_i=D, \blue{\tilde H}$}.
\end{cases}
$$
Here $(2\underline{31})_i(\pi)$ denotes the number of $2\underline{31}$-patterns in $\pi$ with $i$ representing the $2$, i.e.,
$$
(2\underline{31})_i(\pi)=|\{k: k>j\text{ and } \pi_{k+1}<\pi_j=i<\pi_{k}\}|.
$$
\end{lemma}

Given $(M,w)\in\L_n$, let $(M,w)^{rc}=(M',w')$, where
$$
m'_{n+1-i}=\begin{cases}
 m_i \qquad\quad&\text{if $m_i=\red{H},\blue{\tilde H}$},\\
U \qquad&\text{if $m_i=D$},\\
 D \qquad&\text{if $m_i=U$};
\end{cases}
$$
and
$$
w'_{n+1-i}=
\begin{cases}
h_i(M)-w_i, &\quad\text{if $m_i=U, \red{H}$};\\
h_i(M)-1-w_i, &\quad\text{if $m_i=D, \blue{\tilde H}$}.
\end{cases}
$$
\begin{lemma}\label{lem:rcfz}
For any $\pi\in\S_n$, if $\Psi_{FZ}(\pi)=(M,w)$, then $\Psi_{FZ}(\pi^{rc})=(M,w)^{rc}$.
\end{lemma}

\begin{proof}
As $(\underline{31}2)_{n+1-i}(\pi^{rc})=(2\underline{31})_i(\pi)$,
 the result then follows from the construction of $\Psi_{FZ}$ and Lemma~\ref{lem:cor}.
\end{proof}

For any $\pi\in\R_n$, since $\pi$ is rc-invariant and alternating, it follows from Lemma~\ref{lem:rcfz} that $\Psi_{FZ}(\pi)=(D,w)$, where $D$ is a Dyck path (i.e., ballot path ending at $x$-axis) and $(D,w)^{rc}=(D,w)$. If we denote $\Psi(\pi)$ the labeled ballot path by keeping the first half of $(D,w)$, then $\Psi$
establishes a bijection between $\R_n$ and $\P_n$. Composing $\psi$ in Theorem~\ref{thm:psi} with $\Psi$ leads to the following result.

\begin{theorem}\label{thm:chen}
The functional composition $\Psi\circ\psi$ establishes a bijection between $\Sk_n$ and $\P_n$.
\end{theorem}

 For example, if $\pi=2\bar1547\bar6\bar3$, then $\psi(\pi)=5\,2\,14\,11\,12\,7\,9\,6\,8\,3\,4\,1\,13\,10$. After applying $\Psi$ we get the labeled ballot path $(UUUDDUU,0012000)$. 

\section{Concluding remarks}

It should be mentioned that a  bijection different with $\psi\circ\Phi$ between $\W_n$ and $\R_n$ was found in~\cite{Chen}.  Josuat-Verg\`es~\cite{JM} established  another  bijection between $\Sk_n$ and $\P_n$ which does not seem to be  directly related with that one in~\cite{CFJ}. However, it turns out accidentally that our bijection $\Psi\circ\psi$ in Theorem~\ref{thm:chen} is closely related with that one in~\cite{CFJ}, though constructed quite differently. In fact, under the one-to-one correspondence $(P,w)\mapsto(P,\bar w)$ on $\P_n$ with
$$
\bar w_{i}:=
\begin{cases}
h_i(P)-w_i, &\quad\text{if $p_i=U$,}\\
h_i(P)-1-w_i, &\quad\text{if $p_i=D$,}
\end{cases}
$$ our bijection $\Psi\circ\psi$ is the same as that one  between $\Sk_n$ and $\P_n$ introduced in~\cite{CFJ}.

\section*{Acknowledgement} 

S. Chen was partially supported by the NSFC grant (No.~12271511) and the CAS Funds of the Youth Innovation Promotion Association (No.~Y2022001).
Y. Li and Z. Lin were supported by the NSFC grants (Nos.~12271301 \& 12322115) and the Fundamental Research Funds for the Central Universities. S.H.F. Yan was supported by the NSFC grants (Nos.~12471318 \& 12071440).

\end{document}